\documentclass[12pt, oneside, psamsfonts]{amsart}

\newif\ifPDF
\ifx\pdfoutput\undefined\PDFfalse
\else \ifnum \pdfoutput > 0 \PDFtrue
        \else \PDFfalse
        \fi
\fi

\makeatletter
\newcommand{\addresseshere}{%
  \enddoc@text\let\enddoc@text\relax
}
\makeatother

\usepackage[centertags]{amsmath}
\usepackage{amsfonts}
\usepackage{mathrsfs}
\usepackage{textcomp}
\usepackage{amssymb}
\usepackage{amsthm}
\usepackage{newlfont}
\usepackage[all]{xy}
\usepackage[toc,page]{appendix}

\ifPDF
  \usepackage[pdftex]{color, graphicx}
  \usepackage[pdftex, bookmarks, colorlinks]{hyperref}
  \hypersetup{colorlinks=false}

\else
  \usepackage{color}
  \usepackage[dvips]{graphicx}
  \usepackage[dvips]{hyperref}
\fi

\usepackage[scale=0.8]{geometry}

\usepackage[pagewise, mathlines, displaymath]{lineno}

\newtheorem{thm}{Theorem}[section]
\newtheorem{cor}[thm]{Corollary}
\newtheorem{lem}[thm]{Lemma}
\newtheorem{prop}[thm]{Proposition}
\theoremstyle{definition}
\newtheorem{defn}[thm]{Definition}
\newtheorem{ntn}[thm]{Notation}
\theoremstyle{remark}

\newtheorem{example}[thm]{Example}
\numberwithin{equation}{section}

\newcommand{\norm}[1]{\left\Vert#1\right\Vert}
\newcommand{\abs}[1]{\left\vert#1\right\vert}

\newcommand{\Comp}{\mathbb C}

\newcommand{\eps}{\varepsilon}

\newcommand{\Kzero}{\textrm{K}_0}

\begin{document}


\title{A tracially AF algebra which is not $\mathcal Z$-absorbing}

\author{Zhuang Niu}
\address{Department of Mathematics, University of Wyoming, Laramie, WY 82071, USA}
\email{zniu@uwyo.edu}

\author{Qingyun Wang}
\address{Department of Mathematics, University of Oregon, Eugene, OR 97403, USA}
\email{qingyunw@uoregon.edu}

\date{\today}


\begin{abstract}
We show that there is a simple separable unital (non-nuclear but exact) tracially AF algebra $A$ which does not absorb the Jiang-Su algebra $\mathcal Z$ tensorially, i.e., $A \ncong A\otimes\mathcal Z$.
\end{abstract}

\maketitle

\section{Introduction}
Recall that (\cite{LinTAF1}, \cite{LinTAF2}) a unital simple separable C*-algebra $A$ is said to be tracially AF, or TAF, if for any finite set $\mathcal F\subseteq A$, any $\eps>0$, and any $a\in A^+\setminus\{0\}$,  there is a non-zero finite-dimensional C*-subalgebra $F\subseteq A$ such that with $p=1_F$, 
\begin{enumerate}
\item $\norm{fp-pf} < \eps$, $f\in\mathcal F$,
\item $pfp \in_\eps F$, $f\in\mathcal F$, and
\item $1-p$ is Murray-von Neumann equivalent to a projection in $\overline{aAa}$.
\end{enumerate}

 
TAF algebras are relatively well behaved. They always have real rank zero, stable rank one, strict comparison of positive elements, and they are tracially $\mathcal Z$-absorbing (\cite{HO}). The classification of simple separable nuclear TAF algebras which satisfy the Universal Coefficient Theorem (UCT) is one of the milestones in Elliott's classification program for separable nuclear C*-algebras (\cite{LnDuke}), and this class of classifiable TAF algebras coincides with the class of simple AH algebras with real rank zero and with no dimension growth (\cite{EG-RR0AH}). By Corollary 3.1 of \cite{TW1}, it in particular implies that a simple separable nuclear TAF algebra $A$ with the UCT is $\mathcal Z$-absorbing, i.e., $A\cong A\otimes\mathcal Z$, where $\mathcal Z$ is the Jiang-Su algebra.
 

But even without the UCT assumption, Matui and Sato showed that any simple separable nuclear TAF algebra is $\mathcal Z$-absorbing (\cite{Matui-Sato-CP}). In this note, we show that the nuclearity assumption is necessary for the $\mathcal Z$-absorption: there are non-nuclear (but exact---see Appendix) TAF algebras $A$ such that $A \ncong A\otimes\mathcal Z$. Since any tracially AF algebra is tracially $\mathcal Z$-absorbing (see Definition 2.1 of \cite{HO}), this also gives examples of (non-nuclear) tracially $\mathcal Z$-absorbing C*-algebras which are not $\mathcal Z$-absorbing, in contrast to the nuclear case (see Theorem 4.1 of \cite{HO}). (Among many other things, tracial $\mathcal Z$-absorption is also studied in \cite{Fu}).

The main tool that we use is a version of Property $\Gamma$ for C*-algebras. Recall (\cite{GJS-Z}) that a C*-algebra has Property $\Gamma$ if there is a central sequence of unitaries which vanish under all traces. It is a C*-algebra analog of Property $\Gamma$ of a von Neumann factor of type II$_1$. The reduced group C*-algebra of $\mathbf F_2$ (the free group on two generators) does not have the  Property $\Gamma$. In \cite{GJS-Z}, Gong, Jiang and Su showed that all $\mathcal Z$-absorbing C*-algebras have the Property $\Gamma$, and therefore the reduced group C*-algebra of $\mathbf F_2$ is not $\mathcal Z$-absorbing (see Section 2 of \cite{GJS-Z}).

In this note, a modified version of the Property $\Gamma$ is considered (see Definition \ref{prop-gamma}): instead of arbitrary traces, one considers a fixed state. It is shown that for any unital $\mathcal Z$-absorbing C*-algebra and any given state, there exists a central sequence consisting of unitaries which are arbitrarily small under the given state (Corollary \ref{Gamma}). On the other hand, there are TAF algebras in the class constructed by D\u{a}d\u{a}rlat in \cite{MDD-SUBAF} which do not have this property (Proposition \ref{non-Gamma} below), and hence they cannot be $\mathcal Z$-absorbing.

The authors hope that this work could lead to deeper investigations of possible connections between $\mathcal Z$-absorption of (non-nuclear) C*-algebras and central sequences in (type III) von Neumann algebras.

\subsubsection*{Acknowledgements} The research of the first named author is partially supported by a Simons Collaboration Grant (Grant \#317222) and by an NSF grant (DMS-1800882). Part of the results in this paper were obtained during the visits of the second named author to the University of Wyoming in December 2017 and in June 2018, and these visits were also partly supported by the Simons Collaboration Grant \#317222. The second named author would like to thank Chris Phillips for many helpful comments. Both authors are indebted to Caleb Eckhardt for improving the original construction to achieve exactness, see Appendix A. Both authors are also indebted to George Elliott for his careful reading of the paper.

\section{The main result and the proof}

Let $G$ be a countable discrete group. Let $\Comp [G]$, $\mathrm{C}^\ast_{\mathrm{red}}(G)$, and  $\mathrm{C}^\ast(G)$ denote the group algebra, the reduced group C*-algebra, and the full group C*-algebra of $G$ respectively.
The  trace map $ \Comp[ G] \ni a \mapsto a(e) \in \Comp$ can be extended to a tracial state of $\mathrm{C}^\ast_{\mathrm{red}} (G)$, and it is denoted by $\tau$ throughout this paper. 
For $g \in G$, we use $u_g$ to denote the associated standard unitary in
$\mathrm{C}^\ast_{\mathrm{red}} (G)$. We will frequently write $g$ for $u_g$ when there is no confusion.


\subsection{D\u{a}d\u{a}rlat's construction}\label{construction}

The C*-algebras we shall consider in this paper were actually constructed by D\u{a}d\u{a}rlat in \cite{MDD-SUBAF}. We briefly describe the construction for the reader's convenience. 

A C*-algebra is called residually finite-dimensional, or RFD, if it has a separating family of finite-dimensional representations. Let $D$ be a separable unital RFD C*-algebra. Denote by $\pi_1, \pi_2,$... a sequence of finite-dimensional representations of $D$ which separates points, and denote by $n_1, n_2, ...$ the dimension of $\pi_1, \pi_2, ...$, respectively. Denote by $A$ the direct limit of $\mathrm{M}_{k_i}(D)$, where $k_1=1$ and $k_i=(n_1+1) \cdots (n_{i-1}+1)$ for $i = 2, 3, ...$, with the connecting map from $\mathrm{M}_{k_i}(D)$ to $\mathrm{M}_{k_{i+1}}(D)$ defined by
$$a\mapsto \mathrm{diag}(a, \pi_i (a)),\quad a \in \mathrm{M}_{k_i}(D).$$
Then $A$ is a simple unital separable TAF algebra. (See, for instance, Proposition 3.7.8 and Theorem 3.7.9 of \cite{Lin_book} or Example 4.16 of \cite{EN-Tapprox}.)

As a TAF algebra, $A$ has many regularity properties: real rank zero, stable rank one, strict order on projections is determined by traces, and any state on the ordered $\Kzero$-group arises from a trace (\cite{EN-Tapprox}). If $A$ is nuclear, then
$A$ is $\mathcal{Z}$-absorbing, by Theorem 5.4 of \cite{Matui-Sato-CP}. However, this is no longer true without the nuclearity assumption, as is shown by the following result, the main result of the paper.

\begin{thm}\label{main-thm}
There exists a simple separable unital (non-nuclear but exact) tracially AF algebra $A$ which does not absorb the Jiang-Su algebra $\mathcal Z$ tensorially, i.e. $A \ncong A\otimes\mathcal Z$.

More precisely, let $G$ be a discrete group which is not inner amenable (see \cite{Effros} or Definition \ref{def-inn-am} below), and let $D$ be a separable unital RFD C*-algebra such that $\mathrm{C}^\ast_{\mathrm{red}}(G)$ is a quotient of $D$. Denote by $A$ the TAF algebra constructed from $D$ as described above. Then $A$ is not $\mathcal Z$-absorbing, i.e., $A \ncong A\otimes\mathcal Z$. Moreover, with a suitable choice of $D$ and $\mathrm{C}^\ast_{\mathrm{red}}(G)$, the C*-algebra $A$ is exact.
\end{thm}

Let $G$ be a countable discrete group which is not inner amenable. Then, there always exists a (separable unital) RFD C*-algebra $D$ which has $\mathrm{C}^\ast_{\mathrm{red}}(G)$ as a quotient (see Theorem 1.6 of \cite{GM90} or just choose $D$ to be the universal group C*-algebra of $\mathbf F_\infty$,  the free group on countably many generators). Thus the pair $(D, \mathrm{C}^\ast_{\mathrm{red}}(G))$ always exists for any discrete non-inner-amenable group $G$. Moreover, if $G$ is exact, $D$ can be chosen to be exact as well (see Proposition \ref{A1} of Appendix, by Caleb Eckhardt). 

The following are two concrete constructions of the pair $(D, \mathrm{C}^\ast_{\mathrm{red}}(G))$.
\begin{example}\label{example-F2}
Let $G$ be a countable discrete non-inner-amenable group such that $D:=\mathrm{C}^\ast(G)$ is RFD. Then the pair $(D, \mathrm{C}^\ast_{\mathrm{red}}(G))$ satisfies Theorem \ref{main-thm}. One particular example of such a group is $G=\mathbf F_d$, the free group on $d$ generators, where $d=2, 3, ..., \infty$. The group $\mathbf F_d$ is not inner amenable (see \cite{Effros}), and its full group C*-algebra $D$ is RFD by Theorem 7 of \cite{choi-F_2}.

The C*-algebra $A$ constructed from $G$ and $D$ as in \ref{construction} is not exact. In fact, by Theorem 1.1 of \cite{NP-Diagonal}, the group $G$ is maximally almost periodic. Since $G$ is assumed not to be inner amenable, $\mathrm{C^*}(G)$ is not exact by the main theorem of \cite{Harpe-exact}.
\end{example}

\begin{example}
Let $G$ be a countable discrete group which is not inner amenable. Assume that $\mathrm{C}^\ast_{\mathrm{red}}(G)$ is embedded into $\prod_{i=1}^\infty{\mathrm{M}_{n_i}(\Comp)}/\bigoplus_{i=1}^\infty \mathrm{M}_{n_i}(\Comp)$ for some matrix algebra $\mathrm{M}_{n_i}(\Comp)$, $i=1, 2, ...$ (the MF property). Then, the C*-algebra $D:=\pi^{-1}(\mathrm{C}^\ast_{\mathrm{red}}(G))\subseteq \prod{\mathrm{M}_{n_i}(\Comp)}$ is RFD, where $\pi$ is the quotient map $\prod_{i=1}^\infty{\mathrm{M}_{n_i}(\Comp)} \to \prod_{i=1}^\infty{\mathrm{M}_{n_i}(\Comp)}/\bigoplus_{i=1}^\infty \mathrm{M}_{n_i}(\Comp)$.  The pair $(D, \mathrm{C}^\ast_{\mathrm{red}}(G))$ satisfies Theorem \ref{main-thm}.

A particular example is $G= \mathbf F_d$. It follows from Corollary 8.4 of \cite{Haag-Thor} that $\mathrm{C}^\ast_{\mathrm{red}}(\mathbf F_d)$ is MF. Since $\mathrm{C}^\ast_{\mathrm{red}}(\mathbf F_d)$ can be embedded into the nuclear C*-algebra $\mathcal O_2$ (see \cite{Choi79}), $\mathrm{C}^\ast_{\mathrm{red}}(\mathbf F_d)$ is exact, and hence $D$ (as an extension of $\mathrm{C}^\ast_{\mathrm{red}}(\mathbf F_d)$ by $\bigoplus_{i=1}^\infty \mathrm{M}_{n_i}(\Comp)$) and $A$ (constructed above as an inductive limit of matrix algebras over $D$) are exact.
\end{example}

A more interesting example is given by Caleb Eckhardt (Proposition \ref{A2} of Appendix), where an exact RFD algebra $D$ is constructed between  $\mathrm{C}^\ast(\mathbf F_d)$  and $\mathrm{C}^\ast_{\mathrm{red}}(\mathbf F_d)$. Eckhardt also pointed out a general way to produce exact examples (Proposition \ref{A1} of Appendix).

\subsection{Central unitaries in \texorpdfstring{$A$}{A}}
We first introduce the following version of Property $\Gamma$ which is similar to Definition 2.1 of \cite{GJS-Z}: 
\begin{defn}\label{prop-gamma}
Let $A$ be a unital C*-algebra and let $S$ be a collection of states on $A$. We shall say that $A$ has \emph{Property $\Gamma$}
with respect to $S$ if there is a central sequence of unitaries  $(u_i)$ of $A$ such that
$\abs{\rho(u_i)} \to 0$ as $i\to\infty$ for any $\rho \in S$. If $S$ consists of a single state $\rho$, we shall say that $A$ has Property $\Gamma$ with respect to $\rho$.
\end{defn}

 For the C*-algebra $A$ constructed in Theorem \ref{main-thm}, we shall show that there is a state $\rho$ of $A$ such that $A$ does not have Property $\Gamma$ with respect to $\rho$. Let us start with a simple observation.
 

\begin{lem}\label{pre-est-2}
Let $D$ be a unital C*-algebra and let $n$ be a positive integer. 
Let  $$u=\left(\begin{array}{cc} a & b \\ c & d \end{array}\right) \in \mathrm{M}_{1+n}(D)$$ be a matrix over $D$ with $a\in D$, $d\in \mathrm{M}_n(D)$, $b\in \mathrm{M}_{1, n}(D)$, and $c\in \mathrm{M}_{n, 1}(D)$. Then $$\norm{a}, \norm{b}, \norm{c}, \norm{d} \leq \|u\|.$$
\end{lem}
\begin{proof}
Let $p=\mathrm{diag}(1, 0)$ and $q=\mathrm{diag}(0, 1_n)$. Identify $a$ with 
$\left(\begin{array}{cc} a & 0 \\ 0 & 0 \end{array}\right)$ and similarly for $b, c, d$. (This is justified since the identification does not change the norm.) Then
 $$a=pup,\quad b=puq,\quad c=qup,\quad\mathrm{and}\quad d=quq.$$
The lemma follows. 
\end{proof}

\begin{ntn}\label{notation}
Let $G$ be any discrete group and let $m, n$ be natural numbers. We shall use $\norm{\cdot}_{\mathrm{red}}$ to denote the operator norm on $\mathrm{M}_{m, n}(\mathrm{C}_{\mathrm{red}}^\ast(G))$ induced by the C*-norm of 
$\mathrm{C}_{\mathrm{red}}^\ast(G)$. 

Recall that $\tau$ is the canonical trace of $\mathrm{C}_{\mathrm{red}}^\ast(G)$. For $b = (b_1, b_2, ..., b_n) \in \mathrm{M}_{1,n}(\mathrm{C}_{\mathrm{red}}^\ast(G))$, define 
\[
\norm{b}_2 = (\tau(b_1b_1^* + b_2 b_2^* + \cdots + b_n b_n^*))^{\frac{1}{2}} = \tau(bb^*)^{\frac{1}{2}},
\]
and define $\tilde{b}$ to be the function
$$\tilde{b}: G \ni \gamma \mapsto \|b(\gamma)\|_2 = \left(\sum_{i=1}^n \abs{b_{i}(\gamma)}^2 \right)^{\frac{1}{2}}\in\Comp.$$
\end{ntn}

It is straightforward to check that $\|\tilde{b}\|_2 = \|b\|_2$ for any $b \in \mathrm{M}_{1,n}(\mathrm{C}_{\mathrm{red}}^\ast(G))$. In particular, the function $\tilde{b}$ defined above is in $l^2(G)$.

\begin{lem}\label{norms} 
Let $G$ be a discrete group. Let $n \in \mathbb{N}$ and let $a, b$ be
elements in  $\mathrm{M}_{1, n}(\mathrm{C}_{\mathrm{red}}^\ast (G))$. With the notation as \ref{notation}, one has
\begin{enumerate}
    \item \label{p_1} $\widetilde{bu} = \tilde{b}$ for any unitary $u \in \mathrm{M}_n(\Comp)$,
    \item \label{p_2} $\|\tilde{b}\|_2  = \|b\|_2 \leq \|b\|_{\mathrm{red}}$ for any $b \in \mathrm{M}_{1, n}(\mathrm{C}_{\mathrm{red}}^\ast (G))$,
    \item \label{p_3} $\|bu\|_2 = \|b\|_2$ for any unitary $u \in \mathrm{M}_n(\Comp)$,
    \item \label{p_4} $\|g\tilde{a} - \tilde{b}\|_2 \leq \|ga - b\|_2$ for any $g \in G$, and
    \item \label{p_5} $\|b - bd\|_2^2 \leq 2\|b\|_2(\|b\|_2 - \|bd\|_2)$ for any matrix $d \in \mathrm{M}_n(\Comp)$ which is diagonal, positive, and contractive.
\end{enumerate}
\end{lem}

\begin{proof}
(\ref{p_1}) and (\ref{p_2}) follow from straightforward computation. (\ref{p_3}) is a direct consequence of (\ref{p_1}) and (\ref{p_2}). For (\ref{p_4}), note that for each $\gamma \in G$, $a(\gamma)$ is a vector in $\Comp^n$ and $\|a(\gamma)\|_2$ is the vector norm. Using triangle inequality at the third step and (1) at the last step, we have
\begin{align*}
\|g\tilde{a} - \tilde{b}\|_2^2
& =  \sum_{\gamma \in G} \abs{(g\tilde{a})(\gamma) - \tilde{b}(\gamma)}^2 \\ 
& =  \sum_{\gamma \in G} \left(\norm{a(g^{-1}\gamma)}_2 - \norm{b(\gamma)}_2 \right)^2 \\
& \leq  \sum_{\gamma \in G} \norm{a(g^{-1}\gamma)- b(\gamma)}_2^2  \\ 
& =  \sum_{\gamma \in G} \norm{(ga-b)(\gamma)}_2^2 \\
& = \|\widetilde{ga-b}\|_2^2 = \|ga-b\|_2^2.
\end{align*}
For (\ref{p_5}), set $b = (b_1, b_2, ..., b_n)$ and $d = \mathrm{diag}\{\lambda_1, ..., \lambda_n\}$, where $b_i \in \mathrm{C}_{\mathrm{red}}^\ast (G)$ and $\lambda_i \in [0, 1]$ for $i = 1, 2,..., n$. Note that $(1-\lambda_i)^2 \leq (1-\lambda_i)(1+ \lambda_i) = (1 - \lambda_i^2)$, and hence
\begin{align*} 
\|b - bd\|_2^2
& =  \sum_{\gamma \in G}\sum_{i = 1}^n (1 - \lambda_i)^2|b_i(\gamma)|^2 \\ 
& \leq  \sum_{\gamma \in G}\sum_{i = 1}^n (1 - \lambda_i^2)|b_i(\gamma)|^2 = (\|b\|_2^2 - \|bd\|_2^2) \\
& = (\|b\|_2 + \|bd\|_2)(\|b\|_2 - \|bd\|_2) \leq 2\|b\|_2(\|b\|_2 - \|bd\|_2),
\end{align*}
as desired.
\end{proof}

\begin{lem}\label{inv-mean} Let $G$ be a discrete group, and let $n \in \mathbb{N}$. Let $b$ be an element of $\mathrm{M}_{1, n}(\mathrm{C}_{\mathrm{red}}^\ast (G))$ with $\norm{b}_{\mathrm{red}} \leq 1$, and let $g \in G$. Assume that there are $\eps > 0$ and a matrix $\pi(g) \in \mathrm{M}_n(\Comp)$ with norm at most $1$ such that  
\begin{equation}\label{comm-0}
\norm{gb - b\pi(g)}_2 < \eps.
\end{equation}
Then, with notation as \ref{notation}, one has
$\|g\tilde{b} - \tilde{b}\|_2 < \eps + \sqrt{2\eps}$.
\end{lem}
\begin{proof}
Applying the polar decomposition to the matrix $\pi(g)$, we have unitaries $u, w\in \mathrm{M}_n(\Comp)$ and a diagonal matrix $d = \mathrm{diag}\{\lambda_1, ..., \lambda_n\}$, where $\lambda_i \in [0, 1]$, $i = 1, 2,..., n$, such that $$\pi(g) = u(wdw^*).$$

It follows from Lemma \ref{norms}(\ref{p_3}) and \eqref{comm-0} that 
\begin{equation} \label{esti-0}
\norm{g(bw) - buwd}_2  =  \norm{(gb)w - buwd}_2  =  \norm{gb - buwdw^*}_2  < \eps.
\end{equation}

Since $u, w$ are unitary matrices, by Lemma \ref{norms}(\ref{p_3}) again, we have 
\begin{equation}\label{mea-pres}
\norm{buw}_2 = \norm{b}_2= \norm{bw}_2= \norm{g(bw)}_2 \approx_\eps  \norm{buwd}_2.
 \end{equation}
Since $\|b\|_{\mathrm{red}} \leq 1$, It follows from (\ref{p_2}) and (\ref{p_3}) of Lemma \ref{norms}  that
\begin{equation} \label{esti-1}
    \|buw\|_2 = \|b\|_2 \leq \|b\|_{\mathrm{red}} \leq 1.
\end{equation}
Using Lemma \ref{norms}(\ref{p_1}) in the first step, triangle inequality in the second step,
Lemma \ref{norms}(\ref{p_4}) in the third step, (\ref{esti-0}) and Lemma \ref{norms}(\ref{p_5}) in the fourth step and (\ref{mea-pres}) and (\ref{esti-1}) in the final step, we get
\begin{align*}
    \|g\tilde{b} - \tilde{b}\|_2 
& = \|g\widetilde{bw} - \widetilde{buw}\|_2 \\
& \leq \|g\widetilde{bw} - \widetilde{buwd}\|_2 + \|\widetilde{buwd} - \widetilde{buw}\|_2 \\
& \leq \|gbw - buwd\|_2 + \|(buw)d - buw\|_2 \\
& \leq \eps + 2\|buw\|_2^{\frac{1}{2}}(\|(buw)\|_2 - \|buw\|_2)^{\frac{1}{2}} \\
& \leq \eps + 2\sqrt{\eps},
\end{align*}
as desired.
\end{proof}

\begin{lem}\label{small-vec}
Let $ G$ be a countable discrete group which is not amenable. For any $\eps>0$, there exist $\delta>0$ and a finite set $K \subseteq G$ such that if $\xi \in l^2(G)$ satisfies
$$\norm{g\xi - \xi}_2 < \delta,\quad g\in K,$$
then $\norm{\xi}_2 < \eps.$
\end{lem}
\begin{proof}
Let $(K_n)$ be an increasing sequence of finite subsets of $G$ whose union is $G$. If the statement were not true, there will be $\eps_0>0$ such that for any $n \in \mathrm{N}$, there is $\xi_n \in l^2(G)$ with $$\norm{g\xi_n - \xi_n}_2 < \frac{1}{n},\quad g\in K_n,$$ but $$\norm{\xi_n}_2 \geq \eps_0.$$ Then the sequence $\{ \norm{\xi_n}_2^{-1}\xi_n \colon n=1, 2, ... \}$ forms an almost invariant vector for the left regular representation of $G$, which implies that $G$ is amenable, a contradiction.
\end{proof}

The following result is a consequence of Lemma \ref{inv-mean} and Lemma \ref{small-vec}. 
\begin{cor}\label{small-offd}
Let $G$ be a countable discrete non-amenable group. For any $\eps>0$, there exist $\delta>0$ and a finite set $K \subseteq G$ with the following property: 

For any $n \in \mathbb{N}$ and any $b \in \mathrm{M}_{1, n}(\mathrm{C}^\ast_{\mathrm{red}}(G))$, if $\|b\|_{\mathrm{red}} \leq 1$, and if for each $g\in K$ there is a matrix $\pi(g) \in \mathrm{M}_n(\Comp)$ with norm at most $1$ such that
\begin{equation}
\norm{gb - b\pi(g)}_2 < \delta,
\end{equation}
then $\norm{b}_2 < \eps.$
\end{cor}
\begin{proof}
Let $\delta_0>0$ and $K \subseteq G$ denote the constant and finite set provided by Lemma \ref{small-vec} with respect to $\eps$. Pick $\delta> 0$ such that $ \delta + \sqrt{2\delta} \leq \delta_0 $.
With $\tilde{b}$ as \ref{notation}, it follows from Lemma \ref{inv-mean} that, if
$$\norm{gb - b\pi(g)}_{2} < \delta,\quad g\in K,$$ then 
$$\norm{g\tilde{b} - \tilde{b}}_2 < \delta + \sqrt{2\delta} \leq  \delta_0,\quad  g\in K. $$
By the choice of $\delta_0$ and $K$, it follows that $$\norm{b}_2 = \|\tilde{b}\|_2 < \eps,$$ as desired.
\end{proof}

Recall that a \emph{mean} on a countable discrete group $G$ is a positive linear functional $m$ on $l^{\infty}(G)$ with $m(1) = 1$.
Let $e$ be the neutral element of $G$. It is easy to check that the map $d_e \colon l^{\infty}(G) \rightarrow \mathbb{C}$ defined by $d_e(f) = f(e)$, $f \in l^{\infty}(G)$, 
is always a mean, which is called the \emph{trivial mean}.

If $\xi$ is a function on $G$ and $g$ is an element of $G$, define $g \xi g^{-1}$ to be the function 
\[
(g \xi g^{-1})(x) = \xi(g^{-1}xg),\quad x \in G.
\]
\begin{defn}\label{def-inn-am}
(See \cite{Effros}.) A countable discrete group $G$ is said to be \emph{inner amenable} if there is a nontrivial inner invariant mean $m$, in the sense that
\[
m(g \xi g^{-1}) = m( \xi), \quad \xi \in l^{\infty}(G) \text{\,\,and\,\,}  g \in G.
\]
\end{defn}

The following lemma is surely well known. A proof is included for the reader's convenience.

\begin{lem}\label{small-d}
Let $G$ be a countable discrete group which is not inner amenable.
Let $1_e$ denote the identity of $\Comp [G]$. For any $\eps>0$, there are $\delta>0$ and a finite set $K \subseteq G$ such that if $\xi \in \mathrm{C}^\ast_{\mathrm{red}}(G)$ satisfies
$$\norm{g \xi g^{-1} - \xi}_2 < \delta,\quad g\in K,$$
then $\norm{\xi - \tau(\xi)1_e}_2 < \eps.$ 
\end{lem}
\begin{proof}
Assume that the statement were false. 
Choose an increasing sequence of finite subsets $(K_n)$ whose union is $G$. 
Then there is some $\eps_0 > 0$ such that for any $n \in \mathbb{N}$, there is  $\xi_n \in \mathrm{C}^\ast_{\mathrm{red}}(G)$ satisfying
\[
\norm{g \xi_n g^{-1} - \xi_n}_2 < \frac{1}{n},\quad g\in K_n,
\]
but $\|\xi_n - \tau(\xi_n)1_e\|_2 \geq \eps_0$.
Let $\eta_n = \|\xi_n - \tau(\xi_n)1_e\|_2^{-1}(\xi_n - \tau(\xi_n)1_e)$.
Then $\|\eta_n\|_2 = 1$, $\eta_n(e) = 0$, and for any $g\in G$,
\[
\norm{g \eta_n g^{-1} - \eta_n}_2 = \frac{\norm{g \xi_n g^{-1} - \xi_n}_2}{\|\xi_n - \tau(\xi_n)1_e\|_2} \leq \frac{1}{\eps_0 n} 
\rightarrow 0 \quad \text{as\,\,} n \rightarrow \infty.
\]
By the main theorem of \cite{Effros}, this implies that $G$ is inner amenable (the statement of the main theorem of \cite{Effros} assumes that $G$ is i.c.c., but the proof of (2)$\Rightarrow$(3)$\Rightarrow$(4) of the main theorem of \cite{Effros}  does not use this assumption), which contradicts the assumption.
\end{proof}

\begin{defn}\label{defn-seminorm}
Consider the C*-algebra $\mathrm{M}_{1+n}(\mathrm{C}^\ast_{\mathrm{red}}(G))$, where $G$ is a discrete group and $n$ is a natural number, and consider the state $$\rho: \mathrm{M}_{1+n}(\mathrm{C}^\ast_{\mathrm{red}}(G)) \ni \left(\begin{array}{cc} a & b \\ c & d \end{array} \right) \mapsto \tau(a) \in\Comp,$$ where $\tau$ is the canonical trace of $\mathrm{C}^\ast_{\mathrm{red}}(G)$. Define the seminorm $\norm{\cdot}_\rho^\#$ of $\mathrm{M}_{1+n}(\mathrm{C}^\ast_{\mathrm{red}}(G))$ by $$\norm{x}_\rho^\# = (\frac{\rho(xx^*) + \rho(x^*x)}{2})^{\frac{1}{2}} = (\frac{\tau(aa^*) + \tau(bb^*) + \tau(a^*a) + \tau(c^*c)}{2})^{\frac{1}{2}},$$ if $x = \left(\begin{array}{cc} a & b \\ c & d \end{array} \right) \in \mathrm{M}_{1+n}(\mathrm{C}^\ast_{\mathrm{red}}(G)).$
Note that $\norm{\cdot}_\rho^\# \leq \norm{\cdot}_{\mathrm{red}}$.
\end{defn}

\begin{cor}\label{diag}
Let $G$ be a countable discrete group which is not inner amenable.
For any $\eps>0$, there exist $\delta>0$ and a finite set $K \subseteq G$ such that 
for any element $$u=\left(\begin{array}{cc} a & b \\ c & d \end{array} \right) \in \mathrm{M}_{1+n}(\mathrm{C}^\ast_{\mathrm{red}}(G))$$ with $a\in \mathrm{C}^\ast_{\mathrm{red}}(G)$ which satisfies $\norm{u}_{\mathrm{red}}\leq 1$, if for each $g\in K$, there is a matrix $\pi(g) \in \mathrm{M}_{n}(\Comp)$ with norm at most $1$ such that 
\begin{equation}\label{am-comm}
\norm{ \left[\left(\begin{array}{cc} a & b \\ c & d \end{array} \right), \left(\begin{array}{cc} g & 0 \\ 0 & \pi(g) \end{array} \right)\right] }_{\rho}^\#< \delta, \quad  g\in K,
\end{equation} 
then 
$$\norm{b}_2, \norm{c}_2 < \eps$$ and
$$\norm{a-\tau(a)1_e}_2<\eps.$$
\end{cor}

\begin{proof}
Applying Corollary \ref{small-offd} and and Lemma \ref{small-d} to $\eps$, one obtains ($\delta_1$, $K_1$) and ($\delta_2$, $K_2$) respectively. Set $\delta=\frac{1}{\sqrt{2}}\min\{\delta_1, \delta_2\}$ and $K = K_1 \cup K_2$.

Let $\left(\begin{array}{cc} a & b \\ c & d \end{array} \right) \in \mathrm{M}_{1+n}(\mathrm{C}^\ast_{\mathrm{red}}(G))$ satisfy the assumption for this choice of $\delta$ and $K$. It follows from \eqref{am-comm} that
\begin{equation}
\norm{ \left(\begin{array}{cc} ag - ga & b\pi(g) - gb \\ cg - \pi(g)c & d\pi(g) - \pi(g) d \end{array} \right) }_{\rho}^\#< \delta, \quad g\in K,
\end{equation} 
and hence, by the definition of $\norm{\cdot}_\rho^\#$, one has that for any $g\in K$,
\begin{eqnarray}
\norm{ga-ag}_2 & < & \delta_2,\label{comm0} \\
\norm{gb - b\pi(g)}_2 & < & \delta_1,\quad\mathrm{and} \label{comm1}\\
\norm{cg- \pi(g) c}_2 & < & \delta_1. \label{comm2}
\end{eqnarray}

Applying Lemma \ref{small-d} to \eqref{comm0}, one obtains 
$$\norm{a - \tau(a)1_e}_2<\eps.$$

For the estimates on $b$ and $c$, since $\norm{u}_{\mathrm{red}}\leq 1$, by Lemma \ref{pre-est-2}, one has $\|b\|_{\mathrm{red}} \leq 1$ and $\|c\|_{\mathrm{red}} \leq 1$. 
With the choice of $\delta$ and $K$, by \eqref{comm1} and \eqref{comm2}, it follows from Corollary \ref{small-offd} that $$\norm{b}_2<\eps\quad\mathrm{and}\quad\norm{c}_2<\eps,$$
as desired.
\end{proof}

Recall that in D\u{a}d\u{a}rlat's construction (\ref{construction}), the C*-algebra $A$ is the direct limit of $\mathrm{M}_{k_i}(D)$ with the connecting maps
$$a\mapsto \mathrm{diag}(a, \pi_i (a)),$$
and there is a surjective homomorphism $\theta: D \to \mathrm{C}^\ast_{\mathrm{red}}(G)$. Consider the state of $A$ defined by $$\rho_\theta((a_{jk})) = \tau(\theta(a_{11})),\quad (a_{jk}) \in \mathrm{M}_{k_i}(D)$$
where $\tau$ is the canonical trace of $\mathrm{C}^\ast_{\mathrm{red}}(G)$. Similar to Definition \ref{defn-seminorm}, consider the seminorm
$$\norm{a}_{\rho_\theta}^\# := (\frac{\rho_\theta(aa^*) + \rho_\theta(a^*a)}{2})^\frac{1}{2},\quad a\in A.$$

Note that the successive connecting maps $D \rightarrow \mathrm{M}_{k_i}(D)$ always have the form
\begin{equation} \label{embedding}
a \rightarrow \mathrm{diag} (a, \pi(a)),
\end{equation}
where $\pi: D \to \mathrm{M}_{k_i-1}(\Comp 1_D) \subseteq \mathrm{M}_{k_i-1}(D)$ is a finite-dimensional representation of $D$. This induces an embedding of $D$ into $A$. We shall identify $D$ as a sub-C*-algebra of $A$ via this embedding. 

\begin{prop}\label{non-Gamma}
Let $G$ be a countable discrete group which is not inner amenable, and let $D$ be a separable unital RFD algebra  such that $\mathrm{C}^\ast_{\mathrm{red}}(G)$ is a quotient of $D$. Let $A$ be the TAF algebra constructed from $D$ and let $\rho_\theta$ be the state described as above. 

For any $g \in G$, pick an element $\check{g}$ of $D$ with norm 1 which lifts $u_g$, and regard $\check{g}$ as an element of $A$ via the embedding induced by the maps \eqref{embedding}. Then, for any $\eps>0$, there exist $\delta>0$ and a finite set $K \subseteq G$ such that if $u\in A$ is a unitary satisfying $$\norm{u\check{g}-\check{g}u}_{\rho_\theta}^\# < \delta,\quad g\in K,$$ then $$|\rho_\theta(u)| > 1-\eps.$$
In particular, since $\norm{\cdot}_{\rho_\theta}^\# \leq \norm{\cdot}$, the TAF algebra $A$ does not have Property $\Gamma$ with respect to $\rho_\theta$.
\end{prop}
\begin{proof}
Let $\eps>0$ be arbitrary. Choose $\eps_0 > 0$ small enough so that 
\[
\sqrt{1-\eps_0 -2 \eps_0^2} - \eps_0
> 1 - \eps.
\] Let $\delta_0 > 0$ and $K \subseteq G$ be the constant and finite subset provided by Corollary \ref{diag} with $\eps_0$ in the place of $\eps$. Set
\[
\delta = \min \{ \frac{\delta_0}{3}, \frac{\eps_0}{2}\}.
\]
Let $u$ be a unitary in $A$ satisfying $$\|u\check{g}-\check{g}u\|_{\rho_\theta}^\# < \delta,\quad g \in K.$$ By the construction of $A$, there is $n \in \mathbb{N}$ and 
$$v=\left(\begin{array}{cc} a & b \\ c & d \end{array}\right) \in \mathrm{M}_{1+n}(D)\subseteq A$$ with $a\in D$, $\norm{v}=1$, and 
$\|v - u\| < \delta$. 

Note that $\check{g}$ is identified with $\left(\begin{array}{cc} \check{g} & 0 \\ 0 & \pi({\check{g})} \end{array}\right) \in \mathrm{M}_{1+n}(D) \subseteq A,
$
where $\pi(\check{g})$ is a scalar matrix of norm at most 1; so $\theta(\pi(\check{g})) = \pi(\check{g})$.
Since
\begin{eqnarray*}
\norm{ \left[\left(\begin{array}{cc} \theta(a) & \theta(b) \\ \theta(c) & \theta(d) \end{array} \right), \left(\begin{array}{cc} g & 0 \\ 0 & \pi(\check{g}) \end{array} \right)\right] }_{\rho}^\#
& = &\|\theta(\check{g}v -v\check{g})\|_{\rho}^\# = \|\check{g}v-v\check{g}\|_{\rho_\theta}^\# \\
&\leq &\|\check{g}u-u\check{g}\|_{\rho_\theta}^\# + 2\|u - v\| < 3\delta \leq \delta_0, 
\end{eqnarray*}
for $g \in K$, 
by the choice of $\delta_0$ and $K$, it follows from Corollary \ref{diag} that
\begin{equation}\label{eq-sm-1}
\norm{\theta(b)}_2, \norm{\theta(c)}_2 < \eps_0 
\end{equation} and
\begin{equation}\label{eq-sm-2}
\norm{\theta(a)-\tau(\theta(a))1_e}_2<\eps_0.
\end{equation}

Since $\|u -v\| < \delta \leq {\eps_0}/{2}$ and $u$ is a unitary, one has
$$
\norm{(aa^*+bb^*)-1_D}
\leq \norm{vv^* - 1_{\mathrm{M}_{1+n}(D)}}
\leq \norm{uu^* - 1_A} + 2 \|v-u\|
< \eps_0.
$$
Hence by \eqref{eq-sm-1}, \begin{equation}\label{eq-sm-3}\tau(\theta(aa^*))\geq 1 - \eps_0-\tau(\theta(bb^*)) > 1-\eps_0-\eps_0^2.
\end{equation}

On the other hand, write $a=\lambda1_D + a_0$ with $\lambda = \tau(\theta(a))$ and $a_0=a-\tau(\theta(a))1_D$. Then
$$
aa^*  =  (\lambda 1_D + a_0) (\bar{\lambda} 1_D + a^*_0) 
 =  \abs{\lambda}^2 1_D + \lambda a^*_0 + \bar{\lambda} a_0 + a_0a_0^*.
$$
Applying the quotient map $\theta$ and the trace $\tau$ on both sides, by \eqref{eq-sm-2}, we have
$$\tau(\theta(aa^*)) = \abs{\lambda}^2 + \tau(\theta(a_0a_0^*)) = \abs{\lambda}^2 + \norm{\theta(a)-\tau(\theta(a))1_e}_2^2 < \abs{\lambda}^2 + \eps_0^2.$$
Together with \eqref{eq-sm-3}, we have $$\abs{\lambda}^2 > 1-\eps_0-2\eps_0^2,$$
and therefore
$$
|\rho_\theta(u)| \geq |\rho_\theta(v)| - \eps_0
= |\tau(\theta(a))| - \eps_0 > \sqrt{1-\eps_0 -2 \eps_0^2} - \eps_0 > 1 - \eps,
$$
as desired.
\end{proof}

As a simple corollary,  all unitary central sequences in the von Neumann algebra generated by $A$ under the GNS representation associated to $\rho_\theta$ are trivial:
\begin{cor}
Consider the GNS representation of $A$ associated to the state $\rho_\theta$, and consider the von Neumann algebra $A''_{\rho_\theta}$. Then, if there is a sequence of unitaries $(u_n) \subseteq A''_{\rho_\theta}$ such that for any $x\in A_{\rho_\theta}''$,
$[u_n, x] \to 0$ in the strong* topology as $n\to\infty$, one has 
$$(u_n - \rho_\theta(u_n)1_A) \to 0,\quad n\to\infty,$$ in the strong* operator topology.
\end{cor}
\begin{proof}
Since the unitary group of $A$ is strongly* dense in the unitary group of $A_{\rho_\theta}''$ (Kaplansky Density Theorem), there is a sequence of unitaries $(w_n)\subseteq A$ such that $(u_n-w_n) \to 0$ strongly* as $n\to\infty$. Let $x\in A_{\rho_\theta}''$.  Since $[u_n, x] \to 0$ in the strong* topology as $n\to\infty$, one has
\begin{equation*}
\norm{(u_nx - xu_n)(\xi_{1_A})}_{\rho_\theta} + \norm{(u_nx - xu_n)^*(\xi_{1_A})}_{\rho_\theta}\to 0\quad \text{as\,\,} n\to\infty,
\end{equation*}
and hence
\begin{equation*}
\norm{(w_nx - xw_n)(\xi_{1_A})}_{\rho_\theta} + \norm{(w_nx - xw_n)^*(\xi_{1_A})}_{\rho_\theta}\to 0 \quad \text{as\,\,} n\to\infty.
\end{equation*} 
That is,
$$\rho_\theta((w_nx - xw_n)^*(w_nx - xw_n)) + \rho_\theta((w_nx - xw_n)(w_nx - xw_n)^*) \to 0 \quad \text{as\,\,} n\to\infty,$$ and 
$$\norm{w_nx - xw_n}_{\rho_\theta}^\# \to 0 \quad \text{as\,\,} n\to\infty.$$ 
Since $(w_n)\subseteq A$, by Proposition \ref{non-Gamma}, $$\abs{\rho_\theta(w_n)} \to 1 \quad \text{as\,\,} n\to\infty,$$
and this implies
\begin{eqnarray*}
&& \rho_\theta((w_n-\rho_\theta(w_n)1_A)^*(w_n-\rho_\theta(w_n)1_A)) \\
& = & \rho_\theta(w_n^*w_n - \rho_\theta(w_n)w_n^* -\overline{\rho_\theta(w_n)}w_n +\abs{\rho_\theta(w_n)}^21_A) \\
& = & \rho_\theta(1_A) - \abs{\rho_\theta(w_n)}^2 \to 0 \quad \text{as\,\,} n\to\infty.
\end{eqnarray*}
The same calculation also shows  
$$\rho_\theta((w_n-\rho_\theta(w_n)1_A)(w_n-\rho_\theta(w_n)1_A)^*) \to 0 \quad \text{as\,\,} n\to\infty.$$ Therefore
$$\norm{w_n-\rho_\theta(w_n)1_A}_{\rho_\theta}^\# \to 0 \quad \text{as\,\,} n\to\infty,$$ and since $(w_n-u_n) \to 0$ strongly* as $n\to\infty$,
$$\norm{u_n-\rho_\theta(u_n)1_A}_{\rho_\theta}^\# \to 0 \quad \text{as\,\,} n\to\infty.$$

Let $x_1, x_2\in A$ be arbitrary. Set $M = \max \{\norm{x_1}, \norm{x_2}\}$. Then, since $[u_n-\rho_\theta(u_n)1_A, x] \to 0$ strongly* as $n\to\infty$, for any $x\in A$, one has

\begin{eqnarray*}
& &\limsup_{n\to\infty} \left( \norm{(u_n-\rho_\theta(u_n)1_A)x_1}_{\rho_\theta} + \norm{(u_n-\rho_\theta(u_n)1_A)^*x_2}_{\rho_\theta} \right)\\
& \leq & \limsup_{n\to\infty} \left(\norm{x_1} \norm{(u_n-\rho_\theta(u_n)1_A)}_{\rho_\theta} + \norm{x_2} \norm{(u_n-\rho_\theta(u_n)1_A)^*}_{\rho_\theta} \right) \\
& \leq & \limsup_{n\to\infty} M \left(\norm{(u_n-\rho_\theta(u_n)1_A)}_{\rho_\theta} +  \norm{(u_n-\rho_\theta(u_n)1_A)^*}_{\rho_\theta} \right) \\
& \leq & \limsup_{n\to\infty} 2 \sqrt{2} M \left(\norm{(u_n-\rho_\theta(u_n)1_A)}_{\rho_\theta}^{\#} \right) \\
& = & 0.
\end{eqnarray*}

Thus,  $$(u_n - \rho_\theta(u_n) 1_A) \to 0 \quad \text{as\,\,} n\to\infty$$ in the strong* operator topology.
\end{proof}

\subsection{$\mathcal Z$-absorbing C*-algebras have Property $\Gamma$}
Let us show that if a C*-algebra is $\mathcal Z$-absorbing, then it has Property $\Gamma$ (in the sense of Definition \ref{prop-gamma}) with respect to any given state (Corollary \ref{Gamma}).
\begin{prop}\label{Z-unitary}
Let $p, q \in \mathbb{N}$ be prime numbers and let
$Z_{p, q}$ be the dimension drop algebra. Let $\rho$ be a state on $Z_{p, q}$. Then, for any $\eps>0$, there is a unitary $u\in Z_{p, q}$ such that $|\rho(u)| <\eps$.
\end{prop}

\begin{proof}
Recall (see, for instance, \cite{Ell-Cre}) that for a pair of natural numbers $p, q$ which are relatively prime, the dimension drop algebra $Z_{p, q}$ is defined as $$Z_{p, q}:=\{f\in\mathrm{C}([0, 1], \mathrm{M}_{pq}(\Comp)): f(0)\in \mathrm{M}_p(\Comp)\otimes 1_q\ \mathrm{and}\ f(1)\in 1_p\otimes\mathrm{M}_q(\Comp)\}.$$ We assert that the enveloping Borel *-algebra of $Z_{p, q}$ is isomorphic to
$$\mathcal B_{p, q}=\{f\in\mathrm{L}^\infty([0, 1], \mathrm{M}_{pq}(\Comp)): f(0)\in \mathrm{M}_p(\Comp)\otimes 1_q\ \mathrm{and}\ f(1)\in 1_p\otimes\mathrm{M}_q(\Comp)\}.$$ Indeed, denote by $\mathcal B$ the enveloping Borel *-algebra of $Z_{p, q}$. Since $\mathcal B_{p, q}$ is a monotonic sequential closure of $Z_{p, q}$, by Theorem 4.5.9 of \cite{Ped-book}, there is a surjective homomorphism from $\mathcal B$ to $\mathcal B_{p, q}$. Suppose there is an element $a\in\mathcal B$ which is sent to $0$ under this map. Then $a$ must be $0$ under all the irreducible representations of $Z_{p, q}$, and hence $a$  must be $0$ by Corollary 4.5.13 of \cite{Ped-book}. Therefore, the surjection from $\mathcal B$ to $\mathcal B_{p, q}$ is an isomorphism.

Let $\rho$ be a state of $Z_{p, q}$. Then $\rho$ can be extended to a normal state of $\mathcal B_{p, q}$, which will still be denoted by $\rho$.
Identify the center of $\mathcal B_{p, q}$ with
$L^{\infty}([0,1])$.
The restriction of $\rho$ to the center of $\mathcal B_{p, q}$ is then induced by a probability Borel measure $\mu$ on $[0, 1]$; that is,
$$\rho(f) = \int_{[0, 1]} f d \mu,\quad f\in L^\infty([0, 1]) = \mathrm{Z}(\mathcal B_{p, q}).$$

Let $\mathrm{tr}$ denote the tracial state of $\mathrm{M}_{pq}(\Comp)$. Define a (normal) trace of $\mathcal B_{p, q}$ by
$$\phi(f)= \int_{[0, 1]} \mathrm{tr}(f(t))   d\mu(t).$$ We assert that $\rho\ll\phi$. Indeed, if $f \in \mathcal B_{p, q}$ is a positive element such that $\phi(f) = 0$; then, with $E=\{x: f(x) \neq 0\}$, one has that $\mu(E) = 0.$ Set $$\hat{f} = \norm{f}\chi_{E}\in \mathrm{Z}(\mathcal B_{p, q}).$$ It is clear that $f \leq \hat{f}$ and $\rho(\hat{f}) = 0$; hence, $\rho(f) = 0$.

By the Radon-Nikodym Theorem (see, for instance, Theorem 5.3.11 of \cite{Ped-book}), there is a positive (not necessarily bounded) operator $h$ on $H_\phi$ which is affiliated to $\pi_\phi(\mathcal B_{p, q})$  such that 
\begin{equation}\label{RN-thm}
\rho(a) = \left<h\pi_\phi(a)\overline{1_{\mathcal B_{p, q}}}, \overline{1_{\mathcal B_{p, q}}}\right>_\phi = \phi(h \pi_\phi(a)),\quad a\in \mathcal B_{p, q},
\end{equation} 
where $(H_\phi, \pi_\phi)$ is the GNS representation of $\mathcal B_{p, q}$ induced by $\phi$.
For each $t \in \mathbb{R}$, define a real function $f_t$ by $f_t(x) = \min\{x, t\}$, and set $h_t= f_t(h)$. Note that $h_t\in \pi_\phi(\mathcal B_{p, q})$.
Since $\rho(1)=1$, $$1=\phi(h)=\lim_{t\to\infty}\phi(h_t).$$
Thus, for the given $\eps$, there is a sufficiently large $t$ that for any element $a\in \mathcal B_{p, q}$ with $\|a\| \leq 1$,
\begin{equation}\label{cut-der}
\abs{\phi(h\pi_\phi(a)) - \phi(h_t\pi_\phi(a))}^2 = \abs{\phi((h-h_t)^\frac{1}{2}\pi_\phi(a)(h-h_t)^{\frac{1}{2}})}^2\leq \abs{\phi(h-h_t)}^2 \leq \left(\frac{\eps}{3}\right)^2.
\end{equation}


Regarding $h_t$ as an element of $\mathcal B_{p, q}$ (by picking an element of the pre-image), there is $\bar{h}\in (\mathcal B_{p, q})^+$ satisfying
\begin{equation}\label{simple-function} \norm{h_t-\bar{h}}_\infty < \frac{\eps}{3}, \end{equation} 
and 
$\bar{h}$ is a simple function; that is, there are disjoint Borel sets $E_1, E_2, ..., E_n\subseteq [0, 1]$ with $\bigsqcup_{i=1}^n E_i=[0, 1]$ and positive matrices $h_1, h_2, ..., h_n\in\mathrm{M}_{pq}(\Comp)$ such that $$\bar{h}(t) = h_i,\quad\textrm{if $t\in E_i$}.$$

Write $h_i=u^*_id_iu_i$, $i=1, 2, ..., n$, where $u_i$ are unitary matrics and  $d_i$ are diagonal matrices. Note that if $E_i \ni 0$, then $h_i, d_i \in \mathrm{M}_p(\Comp)\otimes 1_q$; and if $E_i \ni 1$, then $h_i, d_i \in 1_p\otimes \mathrm{M}_q(\Comp)$. 
We may require that the unitaries $u_i$ also satisfy the same property: $u_i \in \mathrm{M}_p(\Comp)\otimes 1_q$  if $E_i \ni 0$, and $u_i \in 1_p\otimes \mathrm{M}_q(\Comp)$ if $E_i \ni 1$. Define a unitary $u'\in \mathcal B_{p, q}$ by 
$$u'(t) = u_i^*w_iu_i,\quad \textrm{if $t\in E_i$},$$
where $w_i \in \mathrm{M}_{pq}(\Comp)$ is a unitary with all diagonal elements being zero, $w_i \in \mathrm{M}_p(\Comp)\otimes 1_q$ if $E_i \ni 0$ and $w_i \in 1_p\otimes \mathrm{M}_q(\Comp)$ if $E_i \ni 1$. Then, by \eqref{RN-thm}, \eqref{cut-der}, and \eqref{simple-function}, one has
\begin{eqnarray}
 \rho(u') & = & \phi(h\pi_\phi(u'))\approx_{\frac{\eps}{3}} \phi(h_t u') \approx_{\frac{\eps}{3}} \phi(\bar{h} u') \label{almost-0} \\ 
& = & \sum_{i=1}^n \mathrm{tr}(h_i u_i^*w_iu_i)   \mu(E_i) \nonumber \\
& = & \sum_{i=1}^n \mathrm{tr}(u_i^*d_iu_i u_i^*w_iu_i)   \mu(E_i) \nonumber\\
& = &  \sum_{i=1}^n \mathrm{tr}(d_iw_i)   \mu(E_i) = 0. \nonumber 
\end{eqnarray}

Consider the GNS representation $(\pi_\rho, H_\rho)$ of $Z_{p,q}$. By Corollary 4.5.10 of \cite{Ped-book}, the homomorphism $\pi_\rho$ extends to a normal surjective homomorphism $\pi_\rho'': \mathcal B_{p, q} \to \pi_\rho(Z_{p, q})''$. By the Kaplansky Density Theorem, there is a unitary $v\in \pi_\rho(Z_{p, q})$ such that 
\begin{equation}\label{purt-unitary}
\abs{\left<v \overline{1_{Z_{p, q}}}, \overline{1_{Z_{p, q}}}\right>_\rho - \left<\pi_\rho''(u') \overline{1_{Z_{p, q}}}, \overline{1_{Z_{p, q}}}\right>_\rho} < \frac{\eps}{3}.
\end{equation} 
Since $Z_{p. q}$ has stable rank one, it follows from Proposition 4.3 of \cite{Ror-Ann} that there is a unitary $u\in Z_{p, q}$ such that $\pi_\rho(u) = v$. By \eqref{purt-unitary}, we have $\abs{\rho(u) - \rho(u')} < \frac{\eps}{3}$,
 and hence by \eqref{almost-0}, $|\rho(u)| < \eps$, as desired.
\end{proof}

\begin{cor}\label{Gamma}
Let $A$ be a unital C*-algebra such that $A\cong A\otimes\mathcal Z$, and let $\rho$ be a state of $A$. Then, for any finite set $\mathcal F \subseteq A$ and any $\eps>0$, there is a unitary $u\in A$ such that
$$\norm{ua -au} < \eps,\quad a\in \mathcal F,\quad\textrm{and}\quad |\rho(u)| < \eps.$$
\end{cor}
\begin{proof}
Since $A$ is $\mathcal Z$-absorbing, for any given $\eps>0$ and any finite set $\mathcal F\subseteq A$, there is a unital embedding $ \iota \colon Z_{p, q} \rightarrow A$ such that
$$\norm{a\iota(c) - \iota(c)a} < \eps,\quad a\in\mathcal F 
\text{\,\, and \,\,} c\in Z_{p, q},\ \norm{c}=1.$$
Consider the composition $\rho \circ \iota$, which is a state on $Z_{p, q}$. By Proposition \ref{Z-unitary}, there is a unitary $u\in Z_{p, q}$ satisfying $\abs{(\rho\circ \iota)(u)} < \eps$. Then $\iota(u)$ is a unitary with the desired properties.
\end{proof}

\begin{proof}[Proof of Theorem \ref{main-thm}]
Assume that the TAF algebra $A$ is $\mathcal Z$-absorbing. By Corollary \ref{Gamma}, there is a central sequence consisting of unitaries $(u_n)$ in $A$ with $\rho(u_n) \to 0$. But this contradicts Proposition \ref{non-Gamma} which asserts that $|\rho(u_n)| \to 1$.
\end{proof}



\appendix




\section{{} \\ by Caleb Eckhardt }
In this appendix we point out how to construct exact RFD C*-algebras $D$ that satisfy the conditions of Proposition \ref{non-Gamma} and therefore obtain \emph{exact}, simple separable tracially AF C*-algebras that are not $\mathcal{Z}$-absorbing by Theorem \ref{main-thm}.

It is a well-known corollary of homotopy invariance of quasidiagonality (\cite{Voiculescu91}) that any exact C*-algebra is the quotient of an RFD, exact C*-algebra.
(See Corollary 5.3 of \cite{Nate04}, for example.)
Therefore one immediately obtains the following
\begin{prop}\label{A1}
Let $\Gamma$ be a countable, discrete, exact non-inner-amenable group.  Then there is an exact, unital separable RFD C*-algebra $D$ that maps onto $C^*_{\textup{red}}(\Gamma)$ and subsequently produces an exact, simple separable tracially AF C*-algebra that is not $\mathcal{Z}$-absorbing by Theorem \ref{main-thm}.
\end{prop} 
The point then of this appendix is to point out that in the case of free groups $\mathbf{F}_d$, a minor variant of the above construction  produces an exact RFD C*-algebra $D$ that factors the natural quotient map $\mathrm{C}^\ast(\mathbf F_d)\rightarrow D\rightarrow \mathrm{C}^\ast_{\mathrm{red}}(\mathbf F_d).$ 
The pair $(D, \mathrm{C}^\ast_{\mathrm{red}}(\mathbf F_d))$ satisfies the hypothesis of Theorem \ref{main-thm} and the resulting C*-algebra $A$ is exact.
Furthermore this provides an example of a relatively exotic C*-algebra with good approximation properties. Many examples of exotic group C*-algebras have poor approximation properties---the standard free group examples are neither quasidiagonal nor exact (\cite{Ruan16}).

\begin{prop}\label{A2}
Let $d\geq 2$ and consider the free group $\mathbf{F}_d$ on $d$ generators.  Then there is an exact, RFD C*-algebra $D$ such that the standard quotient map $\mathrm{C}^\ast(\mathbf F_d)\rightarrow \mathrm{C}^\ast_{\mathrm{red}}(\mathbf F_d)$ factors as $\mathrm{C}^*(\mathbf{F}_d)\rightarrow D\rightarrow \mathrm{C}^*_r(\mathbf{F}_d).$ 
\end{prop}
\begin{proof} Choi showed in \cite{Choi79} that $\mathrm{C}^*_{\textup{red}}(\mathbf{F}_d)$ embeds into the (nuclear) Cuntz algebra $\mathcal{O}_2$ and is therefore exact. Let $C$ denote the unitization of $C_0(0,1]\otimes \mathrm{C}^*_{\textup{red}}(\mathbf{F}_d)$ and set $A=M_2(C).$ By standard facts about exact C*-algebras, $A$ is an exact C*-algebra (see \cite{BO-AF-book} for example).  By homotopy invariance of quasidiagonality (\cite{Voiculescu91}) it follows that $A$ is also quasidiagonal.

By a result of Halmos (see \cite[Corollary 7.5.2]{BO-AF-book}  for the statement used below) there is a Hilbert space $H$, a sequence of orthogonal finite rank projections $p_n\in B(H)$ whose sum  increases strongly to the identity, and a C*-algebra $B\subseteq B(H)$ that commutes with each $p_n$ and fits into a split exact sequence 
\begin{equation}\tag{A.3} \label{eq:splitsequence}
0\rightarrow K(H)\rightarrow B+K(H)\rightarrow A\rightarrow 0
\end{equation}
where $K(H)$ denotes the compact operators. Since $A$  and $K(H)$ are exact and the sequence is split, it follows that $B+K(H)$ is exact and therefore so is $B.$ By definition, $B$ is RFD and we may without loss of generality suppose that $p_nK(H)p_n\subseteq B$ for all $n.$

Let $u_1,...,u_d\in \mathrm{C}^*_{\textup{red}}(\mathbf{F}_d)$ be standard generating unitaries.  Then for each $1\leq i\leq d$, the unitary $u_i\oplus u_i^{-1}\in M_2(\mathrm{C}^*_{\textup{red}}(\mathbf{F}_d))$ is homotopic to the identity. Therefore there are unitaries $U_i\in A$ that lift each $u_i\oplus u_i^{-1}.$ Since the sequence in (\ref{eq:splitsequence}) splits we may assume that each $U_i$ belongs $B+K(H).$  Since $p_nK(H)p_n\subseteq B$ for each $n$, it is straightforward to find unitaries $V_i\in B$ that are compact perturbations of $U_i.$

The inclusion map $B\hookrightarrow B+K(H)$ induces an isomorphism 
\[
B/(B\cap K(H))\cong (B+K(H))/K(H)\cong A,
\]
and hence the unitaries $V_i$ are lifts of the $U_i.$  Let $D=C^*(V_1,...,V_d)\subseteq B.$  Since $B$ is exact and RFD, $D$ is also exact and RFD.  
We now have a quotient $D\rightarrow \mathrm{C}^*_{\textup{red}}(\mathbf{F}_d)$ defined by
\begin{equation*}
V_i\in D\mapsto U_i\in A\mapsto \mathrm{diag}(u_i, u_i^{-1})\in M_2(C^*_{\textup{red}}(\mathbf{F}_d))\mapsto u_i\in C^*_{\textup{red}}(\mathbf{F}_d).
\end{equation*}
Since each $V_i$ is unitary,  we obtain the natural factorization $\mathrm{C}^*(\mathbf{F}_d)\rightarrow D\rightarrow \mathrm{C}^*_{\textup{red}}(\mathbf{F}_d).$
\end{proof}



\end{document}